\documentclass[a4paper,11pt,titlepage,twoside]{article}

\usepackage{graphicx}
\usepackage[T1]{fontenc} 
\usepackage[utf8]{inputenc}
\usepackage[english]{babel}
\usepackage{soul}
\usepackage{amsfonts}
\usepackage{amsmath}
\usepackage{amsthm}
\usepackage{amssymb}
\usepackage{mathrsfs}
\usepackage[top=5.5cm, bottom=3cm, left=4cm, right=4cm]{geometry}
\usepackage{setspace}
\usepackage{afterpage}
\usepackage{extarrows}
\usepackage{fancyhdr}
\usepackage{titlesec}
\usepackage{enumitem} \setlist{nosep}
\usepackage[pdftex,breaklinks,colorlinks,linkcolor=blue,anchorcolor=blue]{hyperref}

\theoremstyle{definition}
\newtheorem{defin}{Definition}[section]
\theoremstyle{plain}
\newtheorem{theo}[defin]{Theorem}
\newtheorem{lem}[defin]{Lemma}
\newtheorem{pro}[defin]{Proposition}
\newtheorem{cor}[defin]{Corollary}
\theoremstyle{definition}
\newtheorem{exm}[defin]{Example}

\newtheorem{rems}[defin]{Remarks}
\numberwithin{equation}{section}

\newcommand{\D}{\mathfrak{D}}
\renewcommand{\H}{\mathfrak{H}}

\newcommand{\B}{\mathcal{B}}

\newcommand{\n}[1]{\|#1\|}
\newcommand{\nor}{\|\cdot\|}

\renewcommand{\l}{\langle}
\renewcommand{\r}{\rangle}
\newcommand{\N}{\mathbb{N}}
\newcommand{\R}{\mathbb{R}}
\newcommand{\C}{\mathbb{C}}

\newcommand{\pint}{\l\cdot,\cdot\r}
\newcommand{\pin}[2]{\l#1 , #2\r}

\newcommand{\ol}{\overline}

\newcommand{\sub}{\subseteq}

\newcommand{\mez}{\frac{1}{2}}
\renewcommand{\t}{\mathfrak{t}}
\newcommand{\s}{\mathfrak{s}}
\newcommand{\w}{\mathfrak{w}}
\newcommand{\q}{\mathfrak{q}}
\newcommand{\p}{\mathfrak{p}}

\newcommand{\RR}{\Sigma}

\newcommand{\af}{\mathfrak{a}}
\newcommand{\bb}{\mathfrak{b}}

\renewcommand\labelenumi{\emph{(\roman{enumi})}}
\renewcommand\theenumi\labelenumi

\titleformat{\section}
{\normalfont\fillast \fontsize{12}{15}\scshape}{\thesection.}{0.8em}{}

\titleformat{\subsection}
{\normalfont\fillast \fontsize{11}{12}\scshape}{\thesubsection.}{0.8em}{}

\begin{document}

\thispagestyle{plain}
\begin{center}
	\large
	{\uppercase{\bf A Lebesgue-type decomposition \\ on one side for sesquilinear forms}} \\
	\vspace*{0.5cm}
	{\scshape{Rosario Corso}}
\end{center}

\normalsize 
\vspace*{1cm}	

\small 

\begin{minipage}{11.8cm}
	{\scshape Abstract.} 
Sesquilinear forms which are not necessarily positive may have a different behavior, with respect to a positive form, on each side. For this reason a Lebesgue-type decomposition on one side is provided for generic forms satisfying a boundedness condition. 
\end{minipage}

\vspace*{.5cm}

\begin{minipage}{11.8cm}
\end{minipage}

\vspace*{.5cm}

\begin{minipage}{11.8cm}
\end{minipage}

\vspace*{1cm}
\normalsize

\section{Introduction}

A sesquilinear form $\t$ on a complex vector space $\D$ is a map $\t:\D\times \D\to \C$ which is linear in the first component and anti-linear in the second one. For brevity, we write $\t[f]:=\t(f,f)$ if $f\in \D$. We call $\t$ {\it non-negative} if $\t[f]\geq 0$ for all $f\in \D$ (in symbols, $\t\geq 0$). Given two non-negative forms $\s,\w$ on $\D$ there is a special decomposition of $\s=\s_a+\s_s$ with respect to $\w$, called {\it Lebesgue decomposition} (that goes back to \cite{Simon} and then to \cite{Kos,HSdeS,Gheondea,STT}). Here $\s_a,\s_s$ are non-negative forms on $\D$ which are $\w$-absolutely continuous and $\w$-singular, respectively, according to the following definitions. A non-negative form $\mathfrak{u}$ is 
\begin{itemize}
	\item {\it $\w$-absolutely continuous} (in symbols, $\mathfrak{u}\ll \w$), i.e.\ if $\w[f_n]\to 0$ and $\mathfrak{u}[f_n-f_m]\to 0$ then $\mathfrak{u}[f_n]\to 0$;
	\item {\it $\w$-singular} (in symbols, $\mathfrak{u}\perp \w$), i.e. for all $f\in \D$ there exists $\{f_n\}_{n\in \N}\subset \D$ such that $\w[f_n]\to 0$ and $\mathfrak{u}[f-f_n]\to 0$ or, equivalently, the only non-negative form $\p$ such that $\p\leq \w$ and $\p\leq \mathfrak{u}$ is the null form.  
\end{itemize}
Moreover the forms $\s_a$ and $\s_s$ are singular with respect to each other.

In \cite{Corso_Leb} the assumption of non-negativity for the first form was removed. In other words, a Lebesgue-type decomposition of a sesquilinear form $\t$ (satisfying a boundedness condition) on a complex vector space $\D$ with respect to a non-negative form $\w$ on $\D$ was proved. To explain this decomposition let us recall some notions. We denote by $M(\t)$ the set of non-negative form $\s$ on $\D$ such that $|\t(f,g)|\leq \s[f]^\mez \s[g]^\mez$ for all $f,g\in \D$. Let $\af,\bb$ be non-negative forms on $\D$ with $\af \ll\w$ and $\bb \perp \w$. Consider a form  $\p$ on $\D$ such that for some $\alpha_1,\alpha_2,\alpha_3,\alpha_4\geq 0$ 
\begin{equation}
\label{def_intro}
|\p(f,g)|\leq \alpha_1\af[f]^\mez \af[g]^\mez+\alpha_2\af[f]^\mez \bb[g]^\mez+\alpha_3\bb[f]^\mez \af[g]^\mez+\alpha_4\bb[f]^\mez \bb[g]^\mez, 
\end{equation}
for all $f,g\in \D$. Then a form $\p$ is called {\it $\w$-regular} if \eqref{def_intro} holds with $\alpha_2=\alpha_3=\alpha_4=0$; {\it $\w$-mixed} if \eqref{def_intro} holds with $\alpha_1=\alpha_4=0$; {\it $\w$-strongly singular} if \eqref{def_intro} holds  $\alpha_1=\alpha_2=\alpha_3=0$. \\
The Lebesgue decomposition of a sesquilinear form $\t$ in \cite{Corso_Leb} states that if $M(\t)\neq 0$ then $\t$ is a sum of a $\w$-regular, a $\w$-mixed and a $\w$-strongly singular form on $\D$. Moreover, this three-terms decomposition is minimal: there exist forms that are not the sum of just two of the three types of forms.  

In this paper we provide a new version of Lebesgue decomposition with only two terms (like the classical one for non-negative forms). The new notions that it involves are motivated by the following simple example. 
Let us consider a Hilbert space $\H$ with inner product $\pint$ and $x,y\in \H$ such that $\pin{x}{y}=0$. Define $\w(f,g)=\pin{f}{x}\pin{x}{g}$ and $\t(f,g)=\pin{f}{x}\pin{y}{g}$ for all $f,g\in \H$.  
It is easy to see that the form $\t$ is not $\w$-regular. Nevertheless, we note a `good' behavior of $\t$ with respect to $\w$ on the first component and a `bad' behavior on the second side, in the sense that $|\t(f,g)|\leq \w[f]^\mez \q[g]^\mez$ for all $f,g \in \H$, where $\q(f,g)=\pin{f}{y}\pin{y}{g}$, which is a $\w$-singular non-negative form. 
In particular, for all $g\in \H$ there exists a sequence $\{g_n\}_{n\in \N}$ in $\H$ such that $\w[g_n]\to 0$ and $\t(f,g_n-g)\to 0$.

Taking into account this very simple example, we now introduce the two basic notions of this paper. Let $\t$ be a sesquilinear form  on a complex vector space $\D$. We denote by $M_l(\t)$ the set of forms $\s_1\geq 0$ such that there exists a form $\s_2 \geq 0$ on $\D$ and 
\begin{equation}
\label{dis_intro}
|\t(f,g)|\leq \s_1[f]^\mez \s_2[g]^\mez \qquad \forall f,g \in \D.
\end{equation}
Then $\t$ is called {\it $\w$-left regular} if there exists $\s_1\in M_l(\t)$ such that $\s_1\ll \w$, and {\it $\w$-left strongly singular} if there exists $\s_1\in M_l(\t)$ such that $\s_1\perp \w$. 
Since the adjoint operation $\t \to \t^*$ ($\t^*(f,g)=\ol{\t(g,f)}$) is an involution, we do not introduce the correspondent (and equivalent) definitions on the right side. The presence of a regularity on one side and not in the other one is strictly related to the fact that we deal with generic forms ($\t \ngeq 0$, in general). 

Now we have everything we need to present our main result (Theorem \ref{th_Leb}): a form $\t$ such that $M_l(\t)\neq \varnothing$ is the sum of a  $\w$-left regular form and a $\w$-strongly singular one. 

We analyze the case where the space has finite dimension and take use of this setting to show that the Lebesgue left decomposition is not unique. Moreover, the notions introduced before are connected to simple relations between kernels.

As applications we consider forms induced by measures and give a Lebesgue-type decomposition of bounded operators extending that of Ando \cite{Ando} for non-negative operators. 

Further applications of this paper might also concern extensions of Lebesgue-type decompositions that have been given in other contexts (see 
\cite{Ando_Szy,Gudder,Inoue,Kosaki,Szymanski,Tarcsay}).

\section{Main result}

Before proving the new Lebesgue-type decomposition we recall some notions and fix the notation. 
For a Hilbert space $\H$ we denote by $\B(\H)$ the set of bounded operator everywhere defined on $\H$. We refer to \cite{Kato} for an introduction to sesquilinear forms.

Let $\t$ be a sesquilinear form on $\D$. We put $\ker(\t):=\{f\in \D: \t(f,g)=0, \forall g \in \D\}$, that is a subspace of $\D$.  When $\t$ is non-negative one has $\ker(\t)=\{f\in \D: \t[f]=0\}$ and the following inner product is defined in the quotient  $\D\backslash \ker(\t)$ 
$$
\pin{\pi_\t (f)}{\pi_\t (g)}_\t :=\t(f,g), \qquad f,g\in \D,
$$
where $\pi_\t:\D \to \D\backslash \ker(\t)$ is the canonical projection. 
We denote by $\H_\t$ the completion of $\D\backslash \ker(\t)$ with this inner product. 
Linear combinations of two sesquilinear forms are defined in a natural way.

Now we give more interesting examples of $\w$-left regular and $\w$-left strongly singular forms.

\begin{exm}
	\begin{enumerate}
		\item[(i)] Let $A,B$ be two operators on a Hilbert space $\H$ defined on a subspace $\D$. Then the form $\t(f,g)=\pin{Af}{Bg}$, for $f,g\in \D$, is left regular with respect to the inner product of $\H$ if $A$ is closable. If $A$ has dense kernel, then $\t$ is left strongly singular by \cite[Remark 5.3]{Kos}.
		\item[(ii)]  Let $\H_- \supset \H \supset \H_+$ be a rigged Hilbert space with duality $\pint$ between $\H_-$ and $\H_+$. For $\omega,\varrho\in \H_-$ define the sesquilinear forms
		$$
		\t(\xi,\eta)=\pin{\omega}{\xi}\ol{\pin{\varrho}{\eta}} \;\;\text{ and }\;\; \w(\xi,\eta)=\pin{\xi}{\eta}, \qquad \xi,\eta \in \H_+.
		$$
		Applying the results in  \cite[Examples 1.15, 5.5, 5.9]{Kos}, we find that
		 if $\omega\in \H$, then $\t$ is $\w$-left regular; if $\omega\in\H_-\backslash\H$,  then $\t$ is $\w$-left strongly singular.
		 \item[(iii)] 
		 Let $\Omega\subseteq \R^d$ be an open set, $C^\infty(\Omega)$ the space of infinitely continuously differentiable complex functions on $\Omega$, $\D=C_0^\infty(\Omega)$ the subspace of $C^\infty(\Omega)$ consisting of functions with compact support. We will follow other standard notations: for every multi-index $\alpha\in \N^d$,  $|\alpha|:=\alpha_1+\dots +\alpha_d$ and $\partial^\alpha f :=\frac{\partial^{\alpha_1}}{\partial x_1}\dots \frac{\partial^{\alpha_d}}{\partial x_d}f$. \\
		 Let $n\geq 1$ and, for all $|\alpha|,|\beta|\leq n$, let $r_{\alpha, \beta}\in C^\infty(\Omega)$.  Define the sesquilinear forms for $f,g\in\D$
		 $$
		 \w(f,g)=\int_\Omega f(x)\overline{g(x)}dx, 
		 $$
		 $$
		 \t(f,g)= \sum_{|\alpha|,|\beta|\leq n} \int_\Omega r_{\alpha, \beta}(x)\partial^\alpha f(x) \ol{ \partial^\beta g(x)}dx.
		 $$
		 We denote by $\nor$ the norm of the classical space $L^2(\R)$.  We have 
		 \begin{align*}
		 |\t(f,g)|&\leq \sum_{|\beta|\leq n}  \n{\sum_{|\alpha|\leq n} r_{\alpha,\beta} \partial^\alpha f}\n{\partial^\beta g}\\
		 &\leq N(\sum_{|\beta|\leq n}  \n{\sum_{|\alpha|\leq n} r_{\alpha,\beta} \partial^\alpha f}^2)^\mez (\sum_{|\beta|\leq n}\n{\partial^\beta g}^2)^\mez\\
		 &\leq N\s[f]^\mez \s[g]^\mez
		 \end{align*}
		 for some $N\in \N$ and $\s_1[f]=\sum_{|\beta|\leq n}  \n{\sum_{|\alpha|\leq n} r_{\alpha,\beta} \partial^\alpha f}^2+\sum_{|\beta|\leq n}\n{\partial^\beta f}^2$. By \cite[Lemma 1.13]{Schm}, the operators $\sum_{|\alpha|\leq n} r_{\alpha,\beta} \partial^\alpha$ and $\partial^\beta$ are closable on $\D$. Therefore the non-negative sesquilinear form $\s$ is $\w$-absolutely continuous. It follows that $\t$ is $\w$-regular (in particular $\w$-left regular).
	\end{enumerate}
\end{exm}

\begin{rems}
	\label{rem_intro}
	\begin{enumerate}
		\item[(i)] For a $\w$-left regular form $\t$ we have $\ker(\w)\subseteq \ker(\t)$. This follows from the fact that $\ker(\w)\subseteq \ker(\s)$ for any non-negative $\w$-absolutely continuous form $\s$.
		\item[(ii)] A sesquilinear form $\t$ on $\D$ such that $C\w\in M_l(\t)$, for some $C>0$, is called $\w${\it -left bounded}. Clearly, it is also $\w$-left regular.
		\item[(iii)] For a $\w$-left strongly singular form $\t$ on $\D$ and $f\in \D$ there exists a sequence $\{f_n\}_{n\in \N}\subset \D$ such that $\w[f_n]\to 0$ and $\t(f-f_n,g)\to 0$ for every $g\in \D$. 
	\end{enumerate}
\end{rems}

In the proof of Theorem \ref{th_Leb} we will need the following characterization of singular forms.

\begin{lem}[{\cite[Theorem 6.1]{Kos}}]
	\label{H_s_w_sing}
	A non-negative sesquilinear form $\s$ is $\w$-singular if and only if the map $\pi_{\s+\w}(f)\to (\pi_\s(f),\pi_\w(f))$ is well-defined from $\D\backslash \ker(\s+\w)$ to $(\D\backslash \ker(\s))\times (\D\backslash \ker(\w))$ and extends to an isomorphism between $\H_{\s+\w}$ and the direct product of $\H_\s$ and $\H_\w$
	($\H_{\s+\w}\simeq\H_\s \times \H_\w$).
\end{lem}

\begin{theo}
	\label{th_Leb}
	Let $\t,\w$ be sesquilinear forms on $\D$ such that $M_l(\t)\neq \varnothing$ and $\w\geq 0$. Then $\t=\t_{lr}+\t_{ls}$ where $\t_{lr}$ is a $\w$-left regular form and $\t_{ls}$ is a $\w$-left strongly singular form on $\D$. 
\end{theo}

The proof is similar to the one of \cite{Corso_Leb}. Thus we give only a sketch of it.

\begin{proof}
	Let $\s,\s_2$ be non-negative forms on $\D$ such that $|\t(f,g)|\leq \s[f]^\mez \s_2[g]^\mez$ for all $f,g \in \D$. Let $\s=\s_a+\s_s$ be the classical Lebesgue decomposition of $\s$ with respect to $\w$. 	The forms $\s_a$ and $\s_s$ are singular with respect to each other. By Lemma \ref{H_s_w_sing}, there exists a projection $P\in \B(\H_\s)$ such that 	
	\begin{align*}
	\s_a(f,g)=\pin{P\pi_\s(f)}{P\pi_\s(g)}_\s \; \text{ and }\;
	\s_s(f,g)=\pin{(I-P)\pi_\s(f)}{(I-P)\pi_\s(g)}_\s \nonumber.
	\end{align*}
	Since $\t$ induces a bounded sesquilinear form on $\H_\s\times \H_{\s_2}$ (here $\H_\s$ may be different to $\H_{\s_2}$) there exists a unique bounded operator $T:\H_s\to \H_{\s_2}$, with norm $\n{T}\leq 1$ and
	\begin{equation*}
	\label{eq_t_T}
	\t(f,g)=\pin{T\pi_\s(f)}{\pi_{\s_2}(g)}_{\s_2}, \qquad \forall f,g \in \D.
	\end{equation*}
	Now define  
	\begin{align*}
	\label{t_lr} \t_{lr}(f,g):=\pin{TP\pi_\s(f)}{\pi_{\s_2}(g)}_{\s_2} \;\;\text{ and }\;\;
	\t_{ls}(f,g):=\pin{T(I-P)\pi_\s(f)}{\pi_{\s_2}(g)}_{\s_2} \nonumber
	\end{align*}
	for all $f,g \in \D$. By construction, $\t=\t_{lr}+\t_{ls}$. It is easy to see that $\t_{lr}$ is $\w$-left regular and $\t_{ls}$ is $\w$-left strongly singular. Indeed, by Cauchy-Schwarz inequality, we have for all $f,g\in \D$
	\begin{equation}
	\label{bounds_t_r,s}
	|\t_{lr}(f,g)|\leq \s_a[f]^\mez \s[g]^\mez, \qquad
	|\t_{ls}(f,g)|\leq \s_s[f]^\mez \s[g]^\mez. \qedhere
	\end{equation}
\end{proof} 

A decomposition of $\t$ in $\w$-left regular and $\w$-strongly singular parts like Theorem \ref{th_Leb} is called a {\it Lebesgue left decomposition}. 

We have already said that the adjoint $\t^*$ of $\t$ has corresponding properties on the right side. There are other two classical forms associated to $\t$: the real part $\Re\t=\frac{1}{2}(t+t^*)$ and the imaginary part $\Im\t=\frac{1}{2i}(t-t^*)$. 
In general, however there is no good relations between these notions and the real and imaginary parts of $\t$ as in \cite{Corso_Leb}. We can say that if $\Re\t$ and $\Im \t$ are $\w$-left regular, then so $\t$ is. The converse is not true, $\t$ may be $\w$-left regular but both $\Re\t$ and $\Im \t$ may be not $\w$-left regular.

\section{Examples and applications}

\subsection{Measures}

The Lebesgue decomposition of forms is inspired by the celebrated Lebesgue decomposition of measure. In this subsection we want to give the relations. 

Let $\RR$ be a $\sigma$-algebra on a non-empty set $\mathcal{A}$. We write $\D$ for the complex vector space of simple functions on $(\mathcal{A},\RR)$. We recall some basic notions of measure theory (see for instance \cite{Halmos_m,Rudin}). Let $\mu$ be a (complex) measure on $(\mathcal{A},\RR)$. We said that $\mu$ is {\it non-negative} if $\mu(A)\geq 0$ for all $A\in \RR$.

Given two measures $\mu,\nu$ on $(\mathcal{A},\RR)$ with $\nu$ non-negative, $\mu$ is {\it $\nu$-absolutely continuous}  (in symbol $\mu \ll \nu$) if $\nu(A)=0$ implies $\mu(A)=0$. 
On the other hand, $\mu$ is {\it $\nu$-singular}  (in symbol $\mu \perp \nu$) if there exists $E\in \RR$ such that $\nu(A)=\nu(A\cap E)$ and $\mu(A)=\mu(A\cap E^c)$.

Let $\mu,\nu$ be two measures on $(\mathcal{A},\RR)$ with $\nu$ non-negative. These measures define sesquilinear forms on $\D$ in a natural fashion, i.e.
\begin{equation*}
\label{form_meas}
\t(\phi,\psi)=\int_\mathcal{A} \phi \ol{\psi}d\mu, \qquad
\w(\phi,\psi)=\int_\mathcal{A} \phi \ol{\psi}d\nu.
\end{equation*}
We say that $\t$ is the sesquilinear form {\it induced} by $\mu$. 

The set $M_l(\t)$ is not empty because it contains the form induced by the total variation of $\mu$. 
Moreover, $\t$ is non-negative if and only if $\mu$ is non-negative. In this case, the Lebesgue decompositions of $\t$ and $\mu$ with respect to $\w$ and $\nu$, respectively, are in correspondence. In particular, $\t\ll \w$ if and only if $\mu\ll\nu$; $\t \perp \w$ if and only if $\mu\perp\nu$.

In \cite{Corso_Leb} it was proved that if $\mu$ is $\nu$-singular, then $\t$ is $\w$-strongly singular (therefore $\w$-left strongly singular). In the same way if $\mu$ is $\nu$-absolutely continuous, then $\t$ is $\w$-regular and, in particular, $\w$-left regular. \\
Now we prove the converse of the last statement. If $\t$ is $\w$-left regular then for $A\in \RR$ such that $\nu(A)=0$, we have $\chi_A\in \ker \w \sub \ker \t$ ($\chi_A$ is the characteristic function on $A$). Therefore, $\mu(A)=0$ and $\mu$ is $\nu$-absolutely continuous.

\subsection{Finite dimension case}

It may be interesting to see our setting in the case $\D=\C^n$ for some $n\in \N$. The notions of $\w$-left regular and $\w$-left strongly singular forms turn to be simplified. Anyway the Lebesgue left decomposition is not unique even in this situation. 
As well-known, every form on $\D$ is bounded by the standard norm $\nor$ of $\C^n$ and another advantage is that we can represent a form on $\C^n$ by a $n\times n$ matrix.

\begin{pro}
	\label{lem_dim_finite}
	Let $\t,\w$ be forms on $\C^n$ with $\w\geq 0$.  
	\begin{enumerate}
		\item The following statements are equivalent. 
		\begin{enumerate}
			\item[\emph{(a)}] $\t$ is $\w$-left bounded;
			\item[\emph{(b)}] $\t$ is $\w$-left regular;
			\item[\emph{(c)}] $\ker(\w)\subseteq \ker(\t)$. 
		\end{enumerate}
		\item $\t$ is $\w$-left strongly singular if and only if sum $\ker(\w)+ \ker(\t)=\C^n$ holds.
	\end{enumerate}
\end{pro}
\begin{proof}
	\begin{enumerate}
		\item[(i)]  The implications (a)$\implies$(b)$\implies$(c) hold by Remarks \ref{rem_intro}. Now we prove (c)$\implies$(a). Thus, denoting by $\pi_\w$ the canonical projection $\C^n\to \C^n\backslash \ker(\w)$,  $\widehat{\t}(\pi_\w(f),g):=\t(f,g)$ for $f,g\in \C^n$ defines a sesquilinear form on $\C^n\backslash \ker(\w)\times \C^n$. Obviously,  $\widehat{\t}$ is bounded by the norm of $\C^n\backslash \ker(\w)\times \C^n$, that is $|\t(f,g)|\leq C \w[f]^\mez \n{g}^\mez$ for all $f,g\in \C^n$. 
		\item[(ii)]  Assume $\t$ is $\w$-left strongly singular. By definition there exist $\s_1,\s_2\geq 0, \s_1\perp \w$ such that \eqref{dis_intro} holds. 		
		Let $l,m,p$ be the dimensions of the subspaces $\ker(\s_1)$, $\ker(\w)$, $\ker(\s_1+\w)=\ker(\s_1)\cap \ker(\w)$, respectively.  By Lemma \ref{H_s_w_sing}, $\H_{\s_1+\w}\simeq\H_{\s_1}\times\H_{\w}$ and this gives the following relation between dimensions
		$$
		n-p=\dim \H_{\s_1+\w}=\dim \H_{\s}+\dim \H_{\w}=2n-l-m,
		$$
		i.e. $n=l+m-p=\dim (\ker(\w)+ \ker(\s_1))\leq \dim (\ker(\w)+ \ker(\t))$. Therefore $\ker(\w)+ \ker(\t)=\C^n$.\\
		On the other hand, assume that $\ker(\w)+ \ker(\t)=\C^n$. Put $P$ the projection on $\C^n$ onto $\ker(\t)^\perp$ and $\s(f,g)=\pin{Pf}{Pg}$ for all $f,g\in \C^n$. Thus, for some $C>0$,
		$$
		|\t(f,g)|=|\t(Pf,g)|\leq C\n{Pf}^\mez \n{g}^\mez =C\s[f]^\mez \n{g}^\mez, \qquad \forall f,g\in \C^n.
		$$
		Now we show that $\s\perp \w$.  We have $\ker(\w)+ \ker(\s)=\C^n$. Looking to the dimensions as before, we conclude that $\H_{\s_1+\w}\simeq\H_{\s_1}\times\H_{\w}$ and $\s$ is $\w$-singular by Lemma \ref{H_s_w_sing}.\qedhere
	\end{enumerate} 
\end{proof}

As said before $\t$ is bounded by the norm of $\C^n$. This means that there exists $C>0$ such that $C\iota\in M_l(\t)$, where $\iota$ is the inner product (a non-negative sesquilinear form in fact) of $\C^n$. Therefore we obtain the following special Lebesgue left decomposition. 

\begin{pro}
	Let $\t,\w$ be forms on $\C^n$ with $\w\geq 0$.  Let $\t=\t_{lr}+\t_{ls}$ be the Lebesgue left decomposition of $\t$ with respect to $\w$ taking $C\iota\in M_l(\t)$ for some $C>0$, according to Theorem \ref{th_Leb}. Then $\t_{lr}, \t_{ls}$ are given by 
	\begin{equation}
	\label{parts_dim_finite}
	\t_{lr}(f,g)=\t(Pf,g), \qquad \t_{ls}(f,g)=\t(I-P)f,g)\qquad \forall f,g \in \C^n,
	\end{equation}
	where $P$ is the projection into $\ker(\w)^\perp$. 
\end{pro}
\begin{proof}
We follow the proof of Theorem \ref{th_Leb}. 
Let $\s_2$ be a non-negative form on $\C^n$ such that \eqref{dis_intro} holds with $\s_1=C\iota$. 
The space $\H_{C\iota}$ is essentially $\C^n$. Therefore for some bounded operator $T:\C^n\to \H_{\s_2}$ we have 
\begin{equation*}
\t(f,g)=\pin{Tf}{\pi_{\s_2}(g)}_{\s_2}, \qquad \forall f,g \in \C^n.
\end{equation*}
The Lebesgue decomposition of $\iota$ is simply $\iota_a(f,g)=\pin{Pf}{Pg}$ and $\iota_a(f,g)=\pin{(I-P)f}{(I-P)g}$. 
Thus $\t_{lr}(f,g)=\pin{TPf}{\pi_{\s_2}(g)}_{\s_2}=\t(Pf,g)$ and $\t_{sr}(f,g)=\pin{T(I-P)f}{\pi_{\s_2}(g)}_{\s_2}=\t((I-P)f,g)$ for all $f,g\in \C^n$.
\end{proof}

\begin{exm}
	Here we show that the Lebesgue left decomposition, as constructed in Theorem \ref{th_Leb}, is not unique in general. Indeed, consider the forms $\t,\w,\s_1,\s_2$ on $\C^2$ represented in the usual way by the following matrices with respect to the canonical basis of $\C^2$
	$$
	\begin{pmatrix}
	-1 & 0  \\
	0 & 1  \\
	\end{pmatrix},
	\qquad
	\begin{pmatrix}
	1 & 0  \\
	0 & 0  \\
	\end{pmatrix},
	\qquad	
	\begin{pmatrix}
	1 & 0  \\
	0 & 1  \\
	\end{pmatrix},
	\qquad
	\begin{pmatrix}
	2 & 1  \\
	1 & 2  \\
	\end{pmatrix}.
	$$
	It is easy to see that both $\s_1$ and $\s_2$ belong to $M(\t)$. 
	The Lebesgue left-decomposition of $\t$ with respect to $\w$ and taking $\s_1\in M(\t)$ is $\t=\t_1+\t_2$ where $\t_1,\t_2$ are represented by the following matrices
	$$
	\begin{pmatrix}
	-1 & 0  \\
	0 & 0  \\
	\end{pmatrix},
	\qquad
	\begin{pmatrix}
	0 & 0  \\
	0 & 1  \\
	\end{pmatrix}.
	$$
	On the other hand, the Lebesgue left-decomposition of $\t$ with respect to $\w$ and taking $\s_1\in M(\t)$ is $\t=\t_1+\t_2$ where $\t_1,\t_2$ are represented by the following matrices
	$$
	\begin{pmatrix}
	-1 & -\frac{1}{2}  \\
	0 & 0  \\
	\end{pmatrix},
	\qquad
	\begin{pmatrix}
	0 &  \frac{1}{2} \\
	0 & 1  \\
	\end{pmatrix}.
	$$
	These decompositions can be determined following the proof of Theorem \ref{th_Leb} or imposing the relations between kernels of Proposition \ref{lem_dim_finite}.
\end{exm}

\subsection{Bounded operators}

Let $\H$ be a Hilbert space with inner product $\pint$ and $T,W\in \B(\H)$. We assume that $W$ is non-negative (in symbol $W\geq 0$), i.e. $\pin{Wf}{f}\geq 0$ for all $f\in \H$. \\
Ando \cite{Ando} gave a Lebesgue decomposition for non-negative bounded linear operators, and in \cite{HSdeS} it was shown that this decomposition is actually a particular case of the one for non-negative forms. Inspired by this fact, we give here a Lebesgue-type decomposition of $T$ with respect to $W$, by means of Theorem \ref{th_Leb}. 

First of all, we recall the following definitions from \cite{Ando} and their equivalent formulations by \cite{HSdeS}. Let $S\in \B(\H)$ and $S\geq 0$. Then $S$ is said {\it $W$-absolutely continuous} ($S\ll W$) if one of the following equivalent conditions is satisfied
\begin{enumerate}
	\item[(i)] there exists a sequence of nonnegative operators $\{S_n\}_{n\in \N}\in \B(\H)$ such that $S_m \leq  S_n $ for $n>m$, $S_n\leq c_nW$ for some $c_n>0$ and $\displaystyle \pin{Sf}{f} =\lim_{n \to \infty} \pin{S_nf}{f}$ for all $f\in \H$;
	\item[(ii)]   $\pin{Wf_n}{f_n}\to 0$ and $\pin{S(f_n-f_m)}{(f_n-f_m)}\to 0$ imply $\pin{Sf_n}{f_n}\to 0$.
\end{enumerate}
On the other hand, $S$ is said {\it $W$-singular} ($S\perp W$) if one of the following equivalent conditions holds
\begin{enumerate}
	\item[(i)]  if $Q\in \B(\H)$ and $0\leq Q\leq S$, $0\leq Q\leq W$, then $Q=0$;
	\item[(ii)] for every $f\in \H$ there exists a sequence $\{f_n\}_{n\in \N}\subset \H$ such that $\pin{Wf_n}{f_n}\to 0$ and $\pin{S(f-f_n)}{(f-f_n)}\to 0$.
\end{enumerate}	
Theorem 2 of \cite{Ando} states that for any  non-negative operators $S,W\in \B(\H)$ there exist $S_a,S_s\in \B(\H)$ such that $S_a,S_s\geq 0$, $S_a\ll W, S_s\perp W$ and $S=S_a+S_s$. 
We denote by $M_l(T)$ the set of non-negative operator $S_1\in \B(\H)$ such that there exists a non-negative operator $S_2\in \B(\H)$ such that
\begin{equation}
\label{bound_cond_op}
|\pin{Tf}{g}|\leq \pin{S_1f}{f}^\mez\pin{S_2g}{g}^\mez, \qquad \forall f,g\in \H.
\end{equation}
Note that $M_l(T)$ is never empty because it contains multiplies of the identity. 
We say that $T$ is {\it $W$-left regular} if there exists $S_1\in M_l(T)$ such that $S_1\ll W$, and {\it $W$-left strongly singular} if there exists $S_1\in M_l(T)$ such that $S_1\perp W$.

An operator $T\in \B(\H)$ defines in a natural way a bounded sesquilinear form on $\H$ by $\t(f,g)=\pin{Tf}{g}$ for all $f,g\in \H$. We refer to $\t$ as the sesquilinear form {\it  induced by} $T$. 
For the remark above $M_l(\t)$ is never empty. 
By definition, for $S\in \B(\H)$ with induced form $\s$, we have $S\geq 0$ if and only if $\s\geq 0$. Moreover, $S\ll W$ (resp, $S\perp W$) if and only if $\s\ll \w$ (resp, $\s\perp \w$), and the terms in the Lebesgue decompositions are in correspondence.

We are now ready to give a Lebesgue-type decomposition of an operator $T\in \B(\H)$.

\begin{cor}
	Let $T,W\in \B(\H)$ with $W\geq 0$. Then there exist $T_{lr},T_{ls}\in \B(\H)$ such that $T=T_{lr}+T_{ls}$, $T_{lr}$ is $W$-left regular and $T_{ls}$ is $W$-left strongly singular.
\end{cor}
\begin{proof}
	Let $S_1,S_2\in \B(\H),S_1,S_2\geq 0$ satisfy \eqref{bound_cond_op}. Let $\t,\s,\s_2,\w$ be the induced forms of $T,S_1,S_2,W$, respectively. We follow the construction in the proof of Theorem \ref{th_Leb} and obtain a decomposition $\t=\t_{lr}+\t_{ls}$. From \eqref{bounds_t_r,s} it is simple to see that $\t_{lr},\t_{ls}$ are bounded forms on $\H$. Then we can find $T_{lr},T_{ls}\in \B(\H)$ inducing $\t_{lr},\t_{ls}$, respectively, and  $T=T_{lr}+T_{ls}$. Moreover, if $S_1=S_a+S_s$ is the Lebesgue decomposition of $S$ with respect to $W$, \eqref{bounds_t_r,s} says that for all $f,g\in \H$
	\[
	|\pin{T_{lr} f}{g}|\leq \pin{S_af}{f}^\mez \pin{S_2g}{g}^\mez, \qquad 
	|\pin{T_{ls} f}{g}|\leq \pin{S_sf}{f}^\mez \pin{S_2g}{g}^\mez. \qedhere
	\]
\end{proof}

\section*{Acknowledgments}

This work has been supported by the ``Gruppo Nazionale per l'Analisi Matematica, la Probabilità e le loro Applicazioni'' (GNAMPA – INdAM).

\vspace*{0.5cm}
\begin{center}
	\textsc{Rosario Corso, Dipartimento di Matematica e Informatica} \\
	\textsc{Università degli Studi di Palermo, I-90123 Palermo, Italy} \\
	{\it E-mail address}: {\bf rosario.corso@studium.unict.it}
\end{center}

\end{document}